\documentclass[11pt, a4paper, reqno, english]{amsart}

\usepackage{hyperref}
\usepackage[all]{xy}
\usepackage{amssymb}

\newcommand\FF{\mathbb{F}}
\newcommand\PP{\mathbb{P}}
\newcommand\Pone{{\PP^1}}
\newcommand\Ptwo{{\PP^2}}
\newcommand\NN{\mathbb{N}}
\newcommand\ZZ{\mathbb{Z}}
\newcommand\GG{\mathbb{G}}
\newcommand\Ga{\GG_\mathrm{a}}
\newcommand\Gm{\GG_\mathrm{m}}

\newcommand{\Aone}{{\mathbf A}_1}
\newcommand{\Atwo}{{\mathbf A}_2}
\newcommand{\Athree}{{\mathbf A}_3}
\newcommand{\Afour}{{\mathbf A}_4}
\newcommand{\Afive}{{\mathbf A}_5}
\newcommand{\Aseven}{{\mathbf A}_7}
\newcommand{\Dfour}{{\mathbf D}_4}
\newcommand{\Dfive}{{\mathbf D}_5}
\newcommand{\Dsix}{{\mathbf D}_6}
\newcommand{\Deight}{{\mathbf D}_8}
\newcommand{\Esix}{{\mathbf E}_6}
\newcommand{\Eseven}{{\mathbf E}_7}
\newcommand{\Eeight}{{\mathbf E}_8}
\newcommand{\tS}{{\widetilde S}}

\DeclareMathOperator{\rk}{rk}
\DeclareMathOperator{\Pic}{Pic}
\DeclareMathOperator\Aut{Aut}
\DeclareMathOperator\Bl{Bl}
\DeclareMathOperator\PGL{PGL}
\DeclareMathOperator\GL{GL}
\DeclareMathOperator\diag{diag}

\newcommand\dpbox[2]{#2}
\newcommand\ecbox[2]{*+[F]{#2}}
\newcommand\hsbox[2]{*+[F--]{#2}}

\newtheorem{theorem}{Theorem}
\newtheorem{lemma}[theorem]{Lemma}
\newtheorem{proposition}[theorem]{Proposition}
\theoremstyle{definition}
\newtheorem{definition}[theorem]{Definition}

\begin{document}

\title[Equivariant compactifications]{Equivariant compactifications of two-dimensional algebraic groups}

 \author{Ulrich Derenthal}

 \address{Mathematisches Institut, Ludwig-Maximilians-Universit\"at M\"unchen, 
   Theresienstr. 39, 80333 M\"unchen, Germany}
 
 \email{ulrich.derenthal@mathematik.uni-muenchen.de}

\author{Daniel Loughran}

\address{Department of Mathematics, University Walk, Bristol, UK, BS8 1TW}

\email{daniel.loughran@bristol.ac.uk}

\begin{abstract}
  We classify generically transitive actions of semidirect products $\Ga
  \rtimes \Gm$ on $\Ptwo$. Motivated by the program to study the distribution
  of rational points on del Pezzo surfaces (Manin's conjecture), we determine
  all (possibly singular) del Pezzo surfaces that are equivariant
  compactifications of homogeneous spaces for semidirect products $\Ga \rtimes
  \Gm$.
\end{abstract}



\maketitle

\tableofcontents

\section{Introduction}  

In this note, we are concerned with the classification of algebraic surfaces
that are equivariant compactifications of two-dimensional connected linear
algebraic groups. Over an algebraically closed field $K$ of characteristic
$0$, any such group is isomorphic to the torus $\Gm^2$, the additive group
$\Ga^2$ or a semidirect product $\Ga \rtimes \Gm$.

Here, varieties admitting an action of a connected linear algebraic group $G$
with an open dense orbit are called \emph{equivariant compactifications of
  homogeneous spaces for $G$}. If the stabiliser of a point in the open dense
orbit is trivial, then we simply say that the variety is an \emph{equivariant
  compactification of} $G$.

Equivariant compactifications of tori are widely studied in toric
geometry. The classification of equivariant compactification of additive
groups $\Ga^n$ was initiated by Hassett and Tschinkel \cite{MR1731473}. Here,
we start the classification of equivariant compactifications of semidirect
products $\Ga \rtimes \Gm$. We focus on del Pezzo surfaces (possibly with
rational double points) having such a structure.

This has arithmetic motivations. Namely, the distribution of rational points
on Fano varieties over number fields is predicted by Manin's conjecture
\cite{MR1032922}, giving a precise asymptotic formula for the number of
rational points of bounded height. Using methods of harmonic analysis, it has
been proved for toric varieties \cite{MR1620682}, for equivariant
compactifications of $\Ga^n$ \cite{MR1906155} and recently for certain
equivariant compactifications of $\Ga \rtimes \Gm$ \cite{MR2858922}.

Furthermore, Manin's conjecture is studied systematically in dimension
$2$, where Fano varieties are del Pezzo surfaces, primarily
using universal torsors combined with various analytic techniques. See
\cite[Chapter~2]{MR2559866} for an overview. In the version stated in
\cite{MR1679843}, Manin's conjecture is expected to hold for any del
Pezzo surface whose singularities are rational double points (i.e. canonical);
different behaviour occurs if one allows other singularities (see 
\cite[Example~5.1.1]{MR1679843}).

Therefore, it is important to know which del Pezzo surfaces with at
most rational double points are equivariant compactifications so that
they may be covered by the results from harmonic analysis. It turns
out that this depends only on the \emph{type} of a del Pezzo surface
(which can be expressed by its degree, the types of its singularities
in the $\mathbf{ADE}$-classification and the number of its lines, where
the latter is relevant only in a few cases).

Toric del Pezzo surfaces are easily identified; see
\cite[Figure~1]{math.AG/0604194}, for example. Del Pezzo surfaces that
are equivariant compactifications of $\Ga^2$ were classified in
\cite{MR2753646}.  This leaves the classification of those del Pezzo
surfaces that are equivariant compactifications of semidirect products
$\Ga \rtimes \Gm$, which is the main theorem of this paper.

\begin{theorem} \label{thm:dp}
  A del Pezzo surface $S$, possibly singular with rational double points,
  is an equivariant compactification of some semidirect product
  $\Ga \rtimes \Gm$ if and only if it has one of the following types:
  \begin{itemize}
  \item degree $\ge 7$: all types,
  \item degree $6$: types $\Atwo+\Aone$, $\Atwo$, $2\Aone$, $\Aone$ (with three
    or four lines),
  \item degree $5$: types $\Athree$, $\Atwo+\Aone$, $\Atwo$,
  \item degree $4$: types $\Athree+2\Aone$, $\Dfour$, $\Athree+\Aone$.
  \end{itemize}
  Additionally, precisely the following types are equivariant
  compactifications of a homogeneous space for some semidirect product $\Ga \rtimes \Gm:$
  \begin{itemize}
  \item degree $5$: type $\Afour$,
  \item degree $4$: type $\Dfive$, $\Afour$,
  \item degree $3$: type $\Esix$, $\Afive+\Aone$.
  \end{itemize}
\end{theorem}

Theorem \ref{thm:dp} is visualised diagrammatically in Figure
\ref{fig:blow-ups}.  Note that as remarked in \cite[Section~2]{MR1731473}, if
a variety can be given the structure of a toric variety, this structure is
unique up to equivalence (see Definition~\ref{def:equivalence}).  This may
however fail for other algebraic groups. For example, even $\PP^n$ has
infinitely many different structures as an equivariant compactification of
$\Ga^n$ for $n \geq 6$ \cite[Example~3.6]{MR1731473}.  We consider the
corresponding problem for each semidirect product $\Ga \rtimes \Gm$.  In the
case where $\Ga \rtimes \Gm$ is not the direct product $\Ga \times \Gm$, we
show that up to equivalence $\Ptwo$ admits precisely two different structures
as an equivariant compactification of $\Ga \rtimes \Gm$ (see Theorem
\ref{thm:Ptwo}). We also prove that it admits infinitely many different
structures as an equivariant compactification of a homogeneous space for each
$\Ga \rtimes \Gm$.

Note that a related result is proved in \cite[Section
6]{arXiv:1202.4568}. There however, only the classification of
those equivariant compactifications of homogeneous spaces (``almost
homogeneous" in their terminology) having Picard number one is
considered, while our techniques allow us to identify the equivariant
compactifications of $\Ga \rtimes \Gm$. Moreover, in Section
\ref{Section:dp} we also give results towards classifying the possible
actions which may occur for the surfaces listed in Theorem
\ref{thm:dp}, for example we show which stabilisers may arise.

The layout of this paper is as follows. In Section \ref{Section:groups} we
gather various facts on algebraic group actions and on equivariant
compactifications of homogeneous spaces. In Section \ref{Section:Ptwo} we
classify the different structures that $\Ptwo$ admits as an equivariant
compactification of a homogeneous space for each semidirect product $\Ga
\rtimes \Gm$.  Finally we finish off by considering del Pezzo surfaces and
proving Theorem \ref{thm:dp} in Section \ref{Section:dp}.  Throughout this
paper we work over an algebraically closed field $K$ of characteristic zero
and all algebraic groups will be linear.

\medskip

\noindent\textbf{Acknowledgements:} 
The first-named author was supported by grant DE 1646/2-1 of the
Deutsche Forschungsgemeinschaft and by the Center for Advanced Studies of
LMU M\"unchen. The majority of this work was
completed whilst the second-named author was working at l'Institut de
Math\'{e}matiques de Jussieu and supported by ANR PEPR. The authors
would like to thank Ivan Arzhantsev, Pierre Le Boudec and the referee
for their comments.

\section{Generalities on algebraic groups} \label{Section:groups}
\subsection{Actions of algebraic groups}
We begin by collecting various results on actions of (always linear)
algebraic groups on varieties.
\begin{definition}
  Let $G$ be a connected algebraic group and $X$ a proper normal variety. If
  $X$ admits an action of $G$ that is \emph{generically transitive} (i.e.\
  transitive on some dense open subset), we say that $X$ is an \emph{equivariant
    compactification of a homogeneous space for $G$}. If moreover the action
  is also \emph{generically free} (i.e.\ free on some dense open subset), then we say
  that $X$ is an \emph{equivariant compactification of $G$}.
\end{definition}

For motivation with this terminology, suppose that $X$ is an equivariant
compactification of a homogeneous space for $G$ and let $H$ be the stabiliser
of a general point (i.e.\ a point in the open dense orbit). Then $X$ contains
an open subset isomorphic to the homogeneous space $G/H$ and the action of $G$
on $X$ extends the natural action of $G$ on $G/H$.  If moreover $H$ is
reductive, then the quotient $G/H$ is affine (see
\cite[Theorem~1.1]{MR1304906}) and so the complement of $G/H$ in $X$ is a
divisor \cite[Corollaire~21.12.7]{MR0238860}, which we call the
\emph{boundary} of the action. As example, note that a toric variety is by
definition an equivariant compactification of an algebraic torus. As algebraic
tori are commutative however, every homogeneous space for a torus is in fact
itself a torus, in particular every equivariant compactification of a
homogeneous space for an algebraic torus is also a toric variety. To obtain
homogeneous spaces that are not themselves algebraic groups, one needs to
consider non-commutative groups; we will see many such examples in Section
\ref{Section:dp}.

We will be interested in classifying generically transitive actions up to the
following notion of equivalence.

\begin{definition}\label{def:equivalence}
  Let $G$ be an algebraic group acting on varieties $X_1$ and $X_2$. Then an
  equivalence of (left) $G$-actions is a commutative diagram
  \begin{equation*}
  \begin{split}
    \xymatrix{
      G \times X_1 \ar[d] \ar[r]^{(\alpha,j)} & G \times X_2 \ar[d] \\
      X_1 \ar[r]^j & X_2 }
  \end{split}
  \end{equation*}
  where $\alpha: G \to G$ is an automorphism and $j : X_1 \to X_2$
  is an isomorphism.
\end{definition}
Note that in order to classify generically transitive actions up to
equivalence, we need only consider \emph{left} actions. Indeed, if $G$
acts on the right on a variety $X$ via $(x,g) \mapsto xg$, then we
obtain a left action of $G$ on $X$ defined by $(g,x) \mapsto x
g^{-1}$. This left action is obviously generically transitive
(resp. generically free) if and only if the original action
is. Throughout this paper we will therefore assume that all groups act
on the left. 

Recall that given an action of an algebraic group $G$ on a variety $X$ and a
line bundle $L$ on $X$, a $G$-linearisation of $L$ is a fibrewise linear
action of $G$ on $L$ that respects the action of $G$ on $X$ (see
\cite[Chapter~1]{MR1304906} and \cite[Chapter~7]{MR2004511}).

\begin{lemma} \label{lem:lift} Let $G$ be a connected algebraic group such
  that $\Pic(G)=0$, and suppose that $G$ acts on some normal variety $X$. Then
  every line bundle on $X$ admits a $G$-linearisation.

  In particular for any $n\in \NN$, every projective representation $G \to
  \PGL_n$ admits a lift to a representation $G \to \GL_n$, i.e.\ there exists
  a homomorphism $G \to GL_n$ such that the diagram
  \begin{equation*}
    \begin{split}
      \xymatrix{   G  \ar[dr] \ar[r] & \GL_n \ar[d] \\
        & \PGL_n }
    \end{split}
  \end{equation*}
  is commutative. 
\end{lemma}
\begin{proof}
  By \cite[Theorem~7.2]{MR2004511}, as $G$ is connected we have an exact
  sequence $$\Pic^G(X) \to \Pic (X) \to \Pic(G),$$ where $\Pic^G(X)$ denotes
  the group of isomorphism classes of $G$-linearised line bundles on $X$.  As
  $\Pic(G)=0$, the map $\Pic^G(X) \to \Pic (X)$ is surjective and hence every
  line bundle on $X$ admits a $G$-linearisation.

  To prove the second part of the lemma, note that a projective representation
  $G \to \PGL_n$ gives rise to an action of $G$ on $\PP^{n-1}$.  By the first
  part of the lemma, the line bundle $\mathcal{O}_{\PP^{n-1}}(1)$ admits a
  $G$-linearisation. Therefore we obtain an action on the $n$-dimensional
  vector space $H^0(\PP^{n-1}, \mathcal{O}_{\PP^{n-1}}(1))$, which is the
  required lift to a representation $G \to \GL_n$.
\end{proof}

The algebraic groups of primary interest in this paper (namely
$\Ga,\Gm$ and semidirect products $\Ga \rtimes \Gm$) all have trivial
Picard groups \cite[Remark~7.3]{MR2004511}. Note also that in general,
the choice of linearisation will not be unique if $G$ admits
non-trivial characters (see \cite[(7.3)]{MR2004511}).

Next, we obtain a criterion to help determine whether certain morphisms to
projective space are equivariant.

\begin{lemma}\label{lemma:equivariant}
  Let $X$ be a normal variety together with the action of an algebraic group
  $G$.  Let $L$ be a line bundle on $X$ that is generated by its global
  sections such that $W=H^0(X,L)$ is finite dimensional and which admits a
  $G$-linearisation.  Let $V \subset W$ be a base-point free linear series.
  Then if $\varphi:X \to \PP(V)$ denotes the associated morphism, the
  following are equivalent.
  \begin{enumerate}
  \item $V \subset W$ is invariant under the action of
    $G$. \label{item:linear}
  \item The composed morphism $X \to \PP(V) \subset \PP(W)$ is
    $G$-equivariant. \label{item:equi}
  \item $\PP(V) \subset \PP(W)$ is invariant under the action of
    $G$.\label{item:proj}
  \end{enumerate}
\end{lemma}
\begin{proof}
  The proof that (\ref{item:linear}) implies (\ref{item:equi}) can be found in
  \cite[Section~7.3]{MR2004511}. To show that (\ref{item:equi}) implies
  (\ref{item:proj}), first note that if we let $Y=\varphi(X)$, then $\PP(V)$
  is the only linear subspace of $\PP(W)$ of dimension $n=\dim V +1$ that
  contains $Y$.  Indeed, choose a basis $s_0,\ldots,s_n$ for $V$ and suppose
  that $H \subset \PP(W)$ is another such subspace. Then $H \cap \PP(V)$ is a
  linear subspace of dimension at most $n-1$ containing $Y$, which implies
  that there is a linear relation between $s_0,\ldots,s_n$, giving a
  contradiction. Therefore, as the $G$-equivariance of $\varphi$ implies that
  $g\PP(V)$ contains $Y$ for all $g \in G$, we see that $g\PP(V) = \PP(V)$,
  i.e.\ $\PP(V) \subset \PP(W)$ is invariant under the action of $G$.  This
  proves (\ref{item:proj}).

  Finally, we show that (\ref{item:proj}) implies
  (\ref{item:linear}). The fact that $\PP(V) \subset \PP(W)$ is
  invariant under the action of $G$ implies that for any line $E
  \subset V$ we have $gE \subset V$ for all $g \in G$. Applying this
  to the line spanned by each $s \in V$, we deduce that $gs \in V$ for
  all $g \in G$, which proves (\ref{item:linear}).
\end{proof}

Note that as (\ref{item:proj})  in Lemma
\ref{lemma:equivariant} is independent of the
choice of $G$-linearisation on $L$, we see that (\ref{item:linear}) is
also independent of the choice of $G$-linearisation.
We next consider how the property of being an equivariant
compactification of a homogeneous space behaves with respect to
birational morphisms.

\begin{lemma}\label{lem:blow_up}
  Let $G$ be a connected algebraic group and let $X$ be an equivariant
  compactification of a homogeneous space for $G$.  Let $\pi:\widetilde{X}\to
  X$ be the blow up of $X$ at a subvariety $V \subset X$ that is invariant
  under the action of $G$.  Then $\widetilde{X}$ is an equivariant
  compactification of a homogeneous space for $G$ in such a way that $\pi$ is
  a $G$-equivariant morphism.
\end{lemma}
\begin{proof}
  From the universal property of blow-ups
  \cite[Corollary~II.7.15]{MR0463157}, we obtain a morphism $G \times
  \widetilde{X} \to \widetilde{X}$.  It is easy to see that this gives
  the required action (see the proof of \cite[Lemma~3]{MR2753646}).
\end{proof}

\begin{lemma}\label{lem:blow_down}
  Let $G$ be a connected algebraic group and let $X$ be a smooth equivariant
  compactification of a homogeneous space for $G$. Let $\pi : X \to Y$ be a
  birational morphism to a normal projective variety $Y$. Then $Y$ is an
  equivariant compactification of a homogeneous space for $G$ in such a way
  that $\pi$ is a $G$-equivariant morphism.
\end{lemma}

\begin{proof}
  For equivariant compactifications of $G$, see
  \cite[Proposition~1.3]{MR2858922}. The exact same proof works for
  equivariant compactifications of homogeneous spaces for $G$, as the
  fact that the stabiliser of a general point is trivial is not used
  in the proof.
\end{proof}

Combining these results we obtain the following.

\begin{proposition}\label{prop:desing}
  Let $G$ be a connected algebraic group, $S$ a singular projective normal
  surface and let $\pi:\widetilde{S}\to S$ be a minimal desingularisation.
  Then $S$ is an equivariant compactification of a homogeneous space for $G$
  if and only if $\widetilde{S}$ is, and in which case $\pi$ is a
  $G$-equivariant morphism.
\end{proposition}

\begin{proof}
  The proof of this lemma is essentially the same as the proof of
  \cite[Lemma~4]{MR2753646}.  The fact that $S$ is normal implies that the
  singular locus consists of a finite set of singularities. As $G$ is
  connected, each of these singularities must be fixed under the action of
  $G$. Since the map $\pi$ is given by successively blowing up these
  singularities, on applying Lemma~\ref{lem:blow_up} and
  Lemma~\ref{lem:blow_down} we deduce the result.
\end{proof}

Note that as the $G$-equivariant morphisms in Lemma~\ref{lem:blow_up},
Lemma~\ref{lem:blow_down} and Proposition~\ref{prop:desing} are birational,
they will preserve the order of the stabiliser of each point in the open dense
orbit.

\subsection{Semidirect  products $\Ga \rtimes \Gm$}
We now turn our attention to semidirect products of $\Ga$ and
$\Gm$. Note that one may write down all such groups in a fairly simple
way.  Namely, a semidirect product $\Ga \rtimes \Gm$ is given by a
homomorphism $\Gm \to \Aut(\Ga)=\Gm$. Since homomorphisms $\Gm \to
\Gm$ are given by $t \mapsto t^d$ for any integer $d$, any such
semidirect product has the form $G_d = \Ga \rtimes_{\phi_d} \Gm$ with
$\phi_d(t)(b)=t^db$. The group law on $G_d$ is given
by $$(b,t)\cdot(b',t')=(b+t^db',tt').$$ We keep this notation
throughout this paper. Note that we have obvious isomorphisms $G_d
\cong G_{-d}$ and $G_0 \cong \Ga \times \Gm$.

Later on, we will want to have some information about stabilisers of
generically transitive actions. For this it will be useful to know
which finite subgroups can occur.
\begin{lemma}\label{lem:finite_G_d}
  Any finite subgroup of $G_d$ is conjugate to one of the form
  \begin{align*}
    \mu_n \to G_d, \quad \zeta \mapsto (0,\zeta),
  \end{align*}
  for some $n \in \NN$. Such a subgroup is normal if and only if $n \mid d$,
  in which case $G_d/\mu_n\cong G_{d/n}$.
\end{lemma}
\begin{proof}
  Let $H \subset G_d$ be a finite subgroup. Restricting the exact sequence
  $$0 \to \Ga \to G_d \to \Gm \to 1,$$
  to $H$, we see that $H$ injects into $\Gm$. Indeed as $K$ has characteristic
   zero $\Ga$ has no non-trivial finite subgroups and hence $H \cap \Ga=0$.
   Therefore, there exists $n \in \NN$ such that $H \cong \mu_n$ as an
   algebraic group, in particular $H$ is cyclic and generated by a
   semisimple element.  Such an element is conjugate to one in the
   maximal torus $T= \{ (0,t): t \in \Gm\}$ by
  \cite[Theorem~III.10.6]{MR1102012}. This completes the proof of the
  first part of the lemma.

  A simple calculation shows that $\mu_n$ is not normal if $n \nmid
  d$. If $n \mid d$, then the map
  $$ G_d \to G_{d/n}, \qquad (b,t) \mapsto (b,t^n),$$
  has kernel $\mu_n$ and gives the required isomorphism.
\end{proof}

Note that it follows from Lemma \ref{lem:finite_G_d} that if we wish to
classify generically transitive actions of $G_d$ on a certain surface $S$ for
every $d \in \ZZ$, we may reduce to the case where the action is
\emph{faithful}.  Indeed, as $G_d$ and $S$ have the same dimension the
stabiliser of a general point is finite and hence the kernel of the action
will be a finite normal subgroup.  Quotienting out we obtain a faithful
generically transitive action of $G_{d/n}$ on $S$ for some $n \mid d$.

\section{Actions on the projective plane} \label{Section:Ptwo}
We now classify the generically transitive actions of $G_d$ on $\Ptwo$. 
We begin with a lemma on three-dimensional representations of $\Ga$. 

\begin{lemma}\label{lem:Ga}
  Let $f:\Ga \to \GL_3$ be a faithful representation whose image consists only
  of upper triangular matrices. Then there exist $\alpha_1,\alpha_2,\alpha_3
  \in K$ not all zero such that
  \begin{equation*}
  f(b)=
    \begin{pmatrix}
      1 & \alpha_1b & \alpha_2b + \frac{\alpha_1\alpha_3}{2}b^2\\
      0 & 1 & \alpha_3b \\
      0 & 0 & 1
    \end{pmatrix}.
  \end{equation*}
\end{lemma}
\begin{proof}
  By assumption, we may assume that 
    \begin{equation*}
      f(b)=
    \begin{pmatrix}
      f_{1,1}(b) & f_{1,2}(b) & f_{1,3}(b) \\
      0 & f_{2,2}(b) & f_{2,3}(b) \\
      0 & 0 & f_{3,3}(b)
    \end{pmatrix},
  \end{equation*}
  where all the $f_{i,j}(b)$ are polynomial expressions in $b$.
  For this to define an action we must have 
  \begin{align}
     f_{i,i}(b)\cdot f_{i,i}(b') = f_{i,i}(b+b'),
  \end{align}
  for $i=1,2,3$, i.e.\ each $f_{i,i}$ defines a homomorphism
  $f_{i,i}:\Ga \to \Gm$. Such a homomorphism must be trivial, hence we
  have $f_{i,i}(b)=1$ for each $b \in \Ga$ and $i=1,2,3$.

  Next, differentiating the map $f$ gives an injection of Lie algebras
  $\mathrm{d}f:\mathfrak{g} \to \mathfrak{gl}_3$, where $\mathfrak{g}$
  denotes the Lie algebra of $\Ga$. The morphism $\mathrm{d}f$ sends a
  generator of $\mathfrak{g}$ to a nilpotent matrix
  \begin{equation*}
    \begin{pmatrix}
      0 & \alpha_1 & \alpha_2 \\
      0 & 0 & \alpha_3 \\
      0 & 0 & 0
    \end{pmatrix},
  \end{equation*}
  where $\alpha_1,\alpha_2,\alpha_3 \in K$, at least one of which is non-zero.
  On exponentiating this map we obtain the result.
\end{proof}

The following lemma is the key step in the classification of the
generically transitive actions on $\PP^2$ up to equivalence. It will
also be used later on in our study of such actions on generalised del
Pezzo surfaces.

\begin{lemma}\label{lem:Ptwo}
  Let $d \in \ZZ$ and let $\rho:G_d \to \PGL_3$ be a faithful
  representation whose image consists of only upper triangular
  matrices. Then there exists an element $g \in G_d$ and $k_1,k_2 \in
  \ZZ$ not both zero and $\alpha_1,\alpha_2,\alpha_3 \in K$ not all
  zero such that
  \begin{equation*}
    \rho(g^{-1}(b,t)g) =
    \begin{pmatrix}
      t^{k_1} & \alpha_1 bt^{k_2} & \alpha_2 b+\frac{\alpha_1\alpha_3}{2}
      b^2\\
      0 & t^{k_2} & \alpha_3 b\\
      0 & 0 & 1
    \end{pmatrix}.
  \end{equation*}
  Moreover, the following four conditions must hold:
    \begin{itemize}
  \item $\alpha_1=0$ or $k_1=k_2+d$,
  \item $\alpha_2=0$ or $k_1=d$,
  \item $\alpha_3=0$ or $k_2=d$,
  \item $\alpha_1\alpha_2\alpha_3=0$.
  \end{itemize}
\end{lemma}

\begin{proof}
  Let $U=\{(b,1): b \in \Ga\}$ denote the normal subgroup of $G_d$
  isomorphic to $\Ga$, and let $T$ denote the maximal torus $T= \{
  (0,t): t \in \Gm\}$. The first step of the proof is to analyse the
  behaviour of $\rho$ when restricted to $U$ and $T$. Note that by
  Lemma~\ref{lem:lift}, there exists a lift of $\rho$ to a faithful
  representation $f:G_d \to \GL_3$ that will take the form
  \begin{equation*}
    \begin{pmatrix}
      f_{1,1}(b,t) & f_{1,2}(b,t) & f_{1,3}(b,t) \\
      0 & f_{2,2}(b,t) & f_{2,3}(b,t) \\
      0 & 0 & f_{3,3}(b,t)
    \end{pmatrix},
  \end{equation*}
  where all the $f_{i,j}(b,t)$ are polynomial expressions in $b,t,t^{-1}$.
  For this to define an action, the following relations must hold
  \begin{align} \label{eqn:diagonal}
     f_{i,i}(b,t)\cdot f_{i,i}(b',t') = f_{i,i}(b+t^db',tt'),
  \end{align}
  for $i=1,2,3$. Applying Lemma~\ref{lem:Ga} we see that
  $f_{i,i}(b,0)=1$. Therefore it follows from (\ref{eqn:diagonal})
  that each $f_{i,i}$ defines a homomorphism $f_{i,i}:T \to \Gm$, so
  we must have $f_{i,i}(b,t)=f_{i,i}(0,t)=t^{k_i}$ for some $k_i \in
  \ZZ$ and $i=1,2,3$. Note that we may obviously choose the lift $f$
  so that $k_3=0$. Moreover, we claim that at least one of $k_1$ and
  $k_2$ is non-zero.  Indeed, otherwise $f$ restricted to $T$ would
  give a map $T\to \GL_3$ whose image is unipotent.  As $T\cong \Gm$,
  such a map must be trivial, which contradicts the fact that $f$ is
  faithful.

  Next we find a maximal torus of $G_d$ that has diagonal image under $f$. Let
  $D_3 \subset \GL_3$ denote the subgroup of diagonal matrices and let $H=D_3
  \cap f(G_d)$, which is a closed algebraic subgroup of both $D_3$ and
  $f(G_d)$. Since one of the $k_i$ is non-zero, we see that $H$ is not finite.
  Thus if we let $H^0$ denote the connected component of the identity of $H$,
  it follows that $H^0$ is an algebraic torus of dimension one as it is a
  connected one-dimensional algebraic subgroup of $D_3 \cong \Gm^3$. So $H^0$
  defines a maximal torus in $f(G_d)$, and pulling back via $f$ we obtain a
  maximal torus in $G_d$ with diagonal image. However as any two maximal tori
  are conjugate (see e.g. \cite[Theorem~III.10.6]{MR1102012}), there exists an
  element $g \in G_d$ such that $f(g^{-1}Tg)$ consists of diagonal
  matrices. Moreover by the above we may assume that
  $f(g^{-1}(0,t)g)=\diag(t^{k_1},t^{k_2},1)$.

  Next note that the map $b \mapsto f(g^{-1}(b,1)g)$ is a faithful
  representation of $\Ga$ that consists of upper triangular matrices. Hence
  applying Lemma~\ref{lem:Ga} and using the fact that
  $f(g^{-1}(b,t)g)=f(g^{-1}(b,1)g)f(g^{-1}(0,t)g)$, we see that there exist
  $\alpha_1,\alpha_2,\alpha_3 \in K$ not all zero such that $f(g^{-1}(b,t)g)$
  is given by
  \begin{equation*}
    \begin{pmatrix}
      t^{k_1} & \alpha_1bt^{k_2} & \alpha_2b + \frac{\alpha_1\alpha_3}{2}b^2\\
      0 & t^{k_2} & \alpha_3b \\
      0 & 0 & 1
    \end{pmatrix}.
  \end{equation*}
  One can check that this defines a homomorphism if and only if
  \begin{equation} \label{eq:alpha_i}
     \alpha_1(t^{k_1}-t^{d+k_2})=\alpha_2(t^{k_1}-t^{d})=\alpha_3(t^{k_2}-t^{d})=0,
  \end{equation} 
  for all $t \in K^*$. This gives the list of conditions in the
  lemma. To finish the proof, it suffices to note that if
  $\alpha_1\alpha_2\alpha_3\neq0$ then~(\ref{eq:alpha_i}) implies that
  $k_1=k_2=d=0$, which does not give a faithful representation.
\end{proof}

We are now ready to classify the faithful generically transitive actions of
$G_d$ on $\Ptwo$. We first define the actions that we will be interested in.
Let $d \in \ZZ$ and let $k \in \ZZ \setminus 0$. We define a generically
transitive action of $G_d$ on $\Ptwo$ by
\begin{equation*}
  \tau_{d,k}(b,t) =
    \begin{pmatrix}
      t^{k} & 0 & 0 \\
      0 & t^{d} & b \\
      0 & 0 & 1
    \end{pmatrix}.
\end{equation*}
The following facts are easy to check. We use coordinates $(x:y:z)$ on $\Ptwo$.
\begin{itemize}
\item The stabiliser of a general point has order $|k|$.
\item The representation is faithful if and only if $\gcd(|k|,|d|)=1$.
\item The boundary divisor consists of the two lines $\{x=0\}$ and $\{z=0\}$. 
\item If $k \neq d$, the only fixed points are $(1:0:0)$ and
  $(0:1:0)$. If $k =d$, then the fixed points are exactly the points
  on the line $\{z=0\}$.
\end{itemize}
Note that $\tau_{d,k}$ is not equivalent to $\tau_{d,k'}$ for any
$|k|\neq |k'|$, as the stabilisers of a general point are different in
each case. Also $\tau_{d,k}$ is not equivalent to $\tau_{d,-k}$ for
$d\neq 0$ as these have inequivalent action on the line $\{z=0\}$. One
sees easily however that $\tau_{0,k}$ is equivalent to $\tau_{0,-k}$
on applying the automorphism $(b,t) \mapsto (b,t^{-1})$ of $G_0 = \Ga
\times \Gm$.  We also have another faithful generically transitive
action of $G_d$ on $\Ptwo$ given by
\begin{equation*}
  \rho_d(b,t) =
    \begin{pmatrix}
      t^{2d} & bt^{d} & b^2/2 \\
      0      & t^{d}  & b \\
      0      & 0      & 1
    \end{pmatrix},
\end{equation*}
for any $d\neq0$. Here again it is easy to check the following.
\begin{itemize}
\item The stabiliser of a general point has order $2|d|$.
\item The boundary divisor consists of the line $\{z=0\}$ and the conic $\{y^2=2xz\}$. 
\item The only fixed point is $(1:0:0)$.
\end{itemize}
Note that the boundary divisor for $\rho_d$ does not have strict normal
crossings as the conic lies tangent to the line.  Also, it is easy to see that
$\rho_d$ is not equivalent to $\tau_{d,k}$ for any $kd\neq 0$, as there is no
automorphism of $\Ptwo$ that swaps a line and a conic. Our main theorem in
this section is that any faithful generically transitive action of $G_d$ of
$\Ptwo$ is of the above form, up to equivalence.

\begin{theorem}\label{thm:Ptwo}
  Let $d \neq 0$. Any faithful generically transitive action of $G_d$ on
  $\Ptwo$ is equivalent to either $\tau_{d,k}$ for a unique $k \neq 0$ with
  $\gcd(|k|,|d|)=1$ or $\rho_d$. Any faithful generically transitive action of
  $G_0$ on $\Ptwo$ is equivalent to $\tau_{0,1}$.
\end{theorem}

\begin{proof}
  The action of $G_d$ on $\Ptwo$ gives rise to a faithful
  representation $\rho:G_d \to \PGL_3$. As $G_d$ is solvable, it
  follows from the Lie-Kochin theorem
  \cite[Corollary~III.10.5]{MR1102012} (applied to a lift of $\rho$
  obtained via Lemma \ref{lem:lift}), that we may conjugate by an
  element of $\PGL_3$ to obtain an equivalent action whose image
  consists of only upper triangular matrices. This corresponds to the
  fact that the action on $\Ptwo$ leaves $(1:0:0)$ and $\{z=0\}$
  invariant. Therefore applying Lemma~\ref{lem:Ptwo}, we see that up
  to equivalence $\rho(b,t)$ takes the form
  \begin{equation*}
    \begin{pmatrix}
      t^{k_1} & \alpha_1bt^{k_2} & \alpha_2b + \frac{\alpha_1\alpha_3}{2}b^2\\
      0 & t^{k_2} & \alpha_3b \\
      0 & 0 & 1
    \end{pmatrix}.
  \end{equation*}
  We now proceed by considering the various possibilities on the $\alpha_i$
  given by Lemma~\ref{lem:Ptwo}. First, if $\alpha_1\alpha_2 \neq 0$ then we
  see that $k_2=0,k_1=d$ and $\alpha_3=0$.  This action is not generically
  transitive for any $d$; indeed it preserves the lines $y=\lambda z$ for any
  $\lambda \in K$.

  Next consider the case where $\alpha_1\alpha_3 \neq 0$ and hence
  $\alpha_2=0$. Then Lemma~\ref{lem:Ptwo} implies that $k_1=2d$ and
  $k_2=d$ and therefore $d\neq0$. We claim that this action is
  equivalent to $\rho_d$. Indeed, the conic
  $\{\alpha_1y^2=2\alpha_3xz\}$ is invariant under the action. The
  automorphism of $\Ptwo$ given by $x \mapsto \alpha_3x/\alpha_1$
  moves this conic to the conic $\{y^2=2xz\}$ and gives rise to an
  equivalent action given by
 \begin{equation*}
    \begin{pmatrix}
      t^{2d} & \alpha_3 bt^{d}  & (\alpha_3b)^2/2 \\
      0      & t^{d}    		& \alpha_3b \\
      0      & 0      			& 1
    \end{pmatrix}.
  \end{equation*}
  On performing the automorphism $(b,t) \mapsto (b/\alpha_3,t)$ of $G_d$,
  which rescales $b$, we obtain $\rho_d$.  Thus we may assume that
  $\alpha_1\alpha_3 = 0$ and that $\rho$ takes the form
  \begin{equation*}
    \begin{pmatrix}
      t^{k_1} & \alpha_1bt^{k_2} & \alpha_2b\\
      0 & t^{k_2} & \alpha_3b \\
      0 & 0 & 1
    \end{pmatrix}.
  \end{equation*}
  If $\alpha_2\alpha_3 \neq 0$ then Lemma~\ref{lem:Ptwo} tell us $k_1=k_2=d$
  and $\alpha_1=0$.  Clearly this action is not faithful unless $d = 1$, in
  which case it gives
  \begin{equation*}
    \begin{pmatrix}
      t & 0 & \alpha_2b\\
      0 & t & \alpha_3b \\
      0 & 0 & 1
    \end{pmatrix}.
  \end{equation*}
  The boundary here consists on the lines $\{z=0\}$ and
  $\{\alpha_2y=\alpha_3x\}$. Hence as before, we may perform an automorphism
  of $\Ptwo$ that moves the line $\{\alpha_2y=\alpha_3x\}$ to the line
  $\{x=0\}$ to obtain an action equivalent to $\tau_{1,1}$.

  Thus we have reduced to the case where
  $\alpha_1\alpha_2=\alpha_1\alpha_3=\alpha_2\alpha_3 = 0$.  In particular
  only one of the $\alpha_i$ can be non-zero, and we may even assume that
  $\alpha_i=1$ since applying the automorphism $(b,t) \mapsto (b/\alpha_i,t)$
  of $G_d$ gives an equivalent action. This leaves the following three cases
  \begin{equation*}
    \begin{pmatrix}
      t^{k} & 0 & 0\\
      0 & t^{d} & b \\
      0 & 0 & 1
    \end{pmatrix},\quad
    \begin{pmatrix}
      t^{d} & 0 & b\\
      0 & t^{k} & 0 \\
      0 & 0 & 1
    \end{pmatrix},\quad
    \begin{pmatrix}
      t^{d+k} & bt^{k} & 0\\
      0 & t^{k} & 0 \\
      0 & 0 & 1
    \end{pmatrix}.
  \end{equation*}
  The first action is $\tau_{d,k}$ by definition, whereas the second action is
  seen to be equivalent to $\tau_{d,k}$ on performing the automorphism of
  $\Ptwo$ that swaps $x$ and $y$. As for the third one, we notice that in
  $\PGL_3$ we have
    \begin{equation*}
    \begin{pmatrix}
      t^{d+k} & bt^{k} & 0\\
      0 & t^{k} & 0 \\
      0 & 0 & 1
    \end{pmatrix}=
    \begin{pmatrix}
      t^{d} & b & 0\\
      0 & 1 & 0 \\
      0 & 0 & t^{-k}
    \end{pmatrix},
  \end{equation*}
  which is easily seen to be equivalent to $\tau_{d,-k}$ on performing the
  automorphism of $\Ptwo$ that swaps $y$ and $z$.
\end{proof}

\section{Actions on generalised del Pezzo surfaces} \label{Section:dp}
\subsection{Recap on del Pezzo surfaces}
We now recall various facts that we will need on del Pezzo surfaces,
which can be found for example in \cite{MR579026}, \cite{MR833513} or
\cite{MR940430}.  As before, we work over an algebraically closed
field $K$ of characteristic $0$.

A \emph{generalised del Pezzo surface} $\tS$ is a non-singular projective
surface whose anticanonical class $-K_\tS$ is big and nef.  A normal
projective surface $S$ with ample anticanonical class $-K_S$ is called an
\emph{ordinary del Pezzo surface} if it is non-singular and \emph{singular del
  Pezzo surface} if its singularities are rational double points. Ordinary del
Pezzo surfaces and minimal desingularisations of singular del Pezzo surfaces
are generalised del Pezzo surfaces, and conversely every generalised del Pezzo
surface arises in this way (see \cite[Proposition 0.6]{MR940430}).

The \emph{degree} of a generalised del Pezzo surface $\tS$ is the
self-intersection number $(-K_\tS, -K_\tS)$ of its anticanonical
class. The degree of a singular del Pezzo surface $S$ is defined to be
the degree of its minimal desingularisation.  For $n \in \NN$, a
$(-n)$-curve (or simply a \emph{negative curve}) on a non-singular
projective surface is a rational curve with self-intersection number
$-n$. On generalised del Pezzo surfaces, only $(-1)$- or $(-2)$-curves
may occur (see \cite[page~$29$]{MR940430}).  Moreover, a generalised
del Pezzo surface is ordinary if and only if it contains no
$(-2)$-curves.

A theorem of Demazure (see \cite[Proposition 0.4]{MR940430}) says that
any generalised del Pezzo surface $\tS$ is isomorphic to either
$\PP^2$ (degree $9$), $\PP^1 \times \PP^1$, the Hirzebruch surface
$\FF_2$ (both of degree $8$) or is obtained from $\PP^2$ by a sequence
\begin{equation*}
  \tS = \tS_r \xrightarrow{\rho_r} \tS_{r-1} \to \dots \to \tS_1 \xrightarrow{\rho_1} \tS_0 = \PP^2
\end{equation*}
of $r \le 8$ blow-ups $\rho_i : \tS_i \to \tS_{i-1}$ of points $p_i \in
\tS_{i-1}$ not lying on a $(-2)$-curve on $\tS_{i-1}$, for $i=1, \dots, r$
(with $\tS$ of degree $9-r$).  The Picard group $\Pic(\tS)$ of a generalised
del Pezzo surface is a torsion-free abelian group of rank $10-\deg(\tS)$.  For
a generalised del Pezzo surface $\tS$ of degree $\geq 3$, the anticanonical
linear system defines a birational morphism $\pi : \tS \to S \subset
\PP^{\deg(\tS)}$ to a surface $S$.  If $\tS$ is ordinary, then $\pi$ is in
fact a closed immersion. Otherwise, $\pi$ contracts precisely the
$(-2)$-curves on $\tS$ to the singularities of $S$, and $S$ is a singular del
Pezzo surface with minimal desingularisation $\tS$.

The singularity type of a singular del Pezzo surface $S$ is defined to be
the dual graph of the configuration of $(-2)$-curves on the minimal desingularisation
$\tS$. These graphs are always Dynkin diagrams and are labelled by 
(sums of) $\mathbf{A}_n$ for $n\ge 1$, $\mathbf{D}_n$ for $n
\ge 4$, $\Esix$, $\Eseven$, $\Eeight$.
Moreover in each degree, there are only finitely many possibilities for the
configurations of the negative curves that may occur on generalised del Pezzo surfaces; these
\emph{types} can be distinguished by the $\mathbf{ADE}$-types of the Dynkin
diagrams for the $(-2)$-curves and the number of lines. The latter can be left
out in most cases; exceptions are two $\Athree$-types (with four resp.\ five
lines) in degree $4$ and two $\Aone$-types (with three resp.\ four lines) in
degree $6$; all other exceptions have degree $1$ and $2$ and are not relevant
for us. For some types there are infinitely many isomorphy classes, but for all
types that turn out to be equivariant compactifications of homogeneous spaces
for $\Ga \rtimes \Gm$, we will see that there is precisely one such surface
up to isomorphism.

\subsection{Actions on generalised del Pezzo surfaces}
We now consider the classification of those generalised del Pezzo surfaces
that admit a generically transitive action of $G_d$ for some $d$, with the aim
of proving Theorem~\ref{thm:dp}. It turns out that such surfaces must satisfy
a special geometric condition.

\begin{lemma}\label{lemma:candidates}
  Let $\tS$ be a generalised del Pezzo surface that is an equivariant
  compactification of a homogeneous space for $G_d$ for some $d$.  Then $$\#\{
  \text{negative curves on } \tS\} \leq \rk \Pic \tS + 1.$$
\end{lemma}
\begin{proof}
  First note that if $\tS \cong \Pone \times \Pone$ or $\tS \cong
  \FF_2$, then there is at most one negative curve and the inequality
  trivially holds.  So we may assume that $\tS$ is obtained from $\Ptwo$ by a
  sequence of $r$ blow-ups. To prove the inequality in this case, it suffices
  to show that the boundary of the action consists of $r+2=\rk \Pic \tS +1$
  irreducible curves. Indeed, let $E$ be a negative curve on $\tS$. By
  Lemma~\ref{lem:lift}, the line bundle $\mathcal{O}_{\tS}(E)$ admits a
  $G_d$-linearisation, in particular the divisor class of $E$ is invariant
  under the action of $G_d$. As $E$ is the unique effective curve in its
  divisor class, we see that $E$ itself is invariant under the action of $G_d$
  and therefore $E$ must lie on the boundary. The fact that the boundary
  consists of $r+2$ irreducible curves then gives the required inequality.

  To prove the claim we proceed by induction. Let $X$ be a smooth projective
  equivariant compactification of a homogeneous space for $G_d$ that contains
  a $(-1)$-curve $E$ and let $\pi:X \to Y$ be the map given by contracting
  $E$. Note that we may assume that $Y$ is an equivariant compactification of
  a homogeneous space for $G_d$ and that $\pi$ is $G_d$-equivariant by
  Lemma~\ref{lem:blow_down}. As $\pi$ is an isomorphism outside $E$, we see
  that $X$ has exactly one more boundary component than $Y$. Applying this
  inductively to $\tS$, we see that the boundary of the action on $\tS$
  consists of $r+n$ irreducible curves, where $n$ is the number of irreducible
  curves on the boundary of the action on $\Ptwo$. However, by the
  classification given in Theorem~\ref{thm:Ptwo} we know that $n=2$. This
  proves the claim and hence completes the proof of the lemma.
\end{proof}

From the classification of generalised del Pezzo surfaces that can be
found in \cite{math.AG/0604194}, for example, it is straightforward to
write down the list of surfaces that satisfy the condition of
Lemma~\ref{lemma:candidates}. These are shown in
Figure~\ref{fig:blow-ups}.
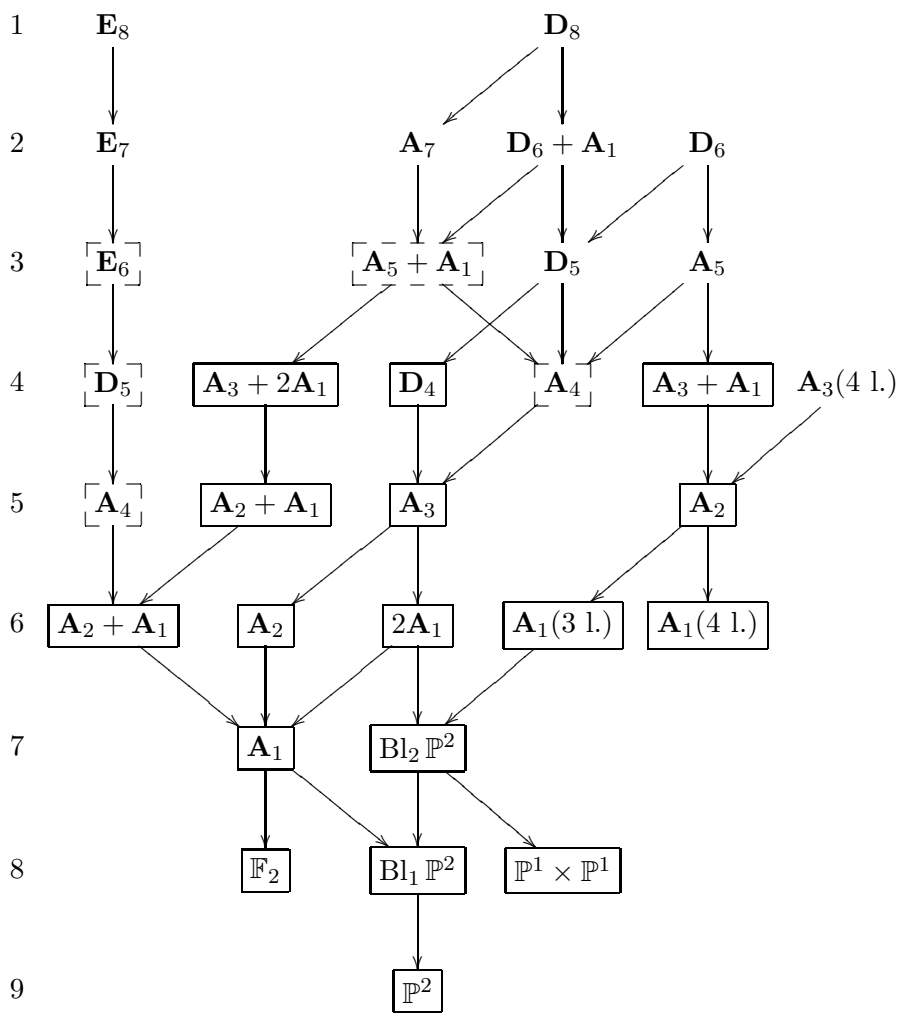
\begin{figure}[ht]
  \begin{equation*}
    \xymatrix@R=.4in @C=0.08in{
   1  &    \dpbox{1}{\Eeight} \ar[d]   &      &      &              \Deight \ar[d] \ar[ld]   &      &      &  \\
   2   &    \dpbox{2}{\Eseven}       \ar[d]    &      &    \Aseven \ar[d]   &              \Dsix+\Aone \ar[d] \ar[ld]    &    \Dsix \ar[d] \ar[ld]   &      &     \\
   3   &    \hsbox{3}{\Esix} \ar[d]    &       &          \hsbox{3}{\Afive+\Aone} \ar[ld] \ar[rd]    &    \Dfive \ar[d] \ar[ld]   &    \Afive  \ar[d] \ar[ld]    &      &     \\
   4   &    \hsbox{4}{\Dfive}       \ar[d]    &    \ecbox{6}{\Athree+2\Aone}  \ar[d]  &    \ecbox{6}{\Dfour} \ar[d]    &    \hsbox{4}{\Afour} 	  	  \ar[ld]    &    \ecbox{6}{\Athree +\Aone} \ar[d]    &    \Athree (4 \text{ l.}) \ar[ld]   &     \\
   5   &    \hsbox{5}{\Afour} \ar[d]    &   \ecbox{6}{\Atwo+\Aone} \ar[ld]     &    \ecbox{6}{\Athree} \ar[d]  \ar[ld]  &        &    \ecbox{6}{\Atwo}  \ar[d] \ar[ld]    &       &     \\
   6   &    \ecbox{6}{\Atwo+\Aone} \ar[rd]    &   \ecbox{6}{\Atwo} \ar[d]    &          \ecbox{6}{2\Aone} \ar[ld] \ar[d]    &    \ecbox{6}{\Aone (3 \text{ l.})} \ar[ld]    &    \ecbox{6}{\Aone (4 \text{ l.})}  &       &     \\
   7   &       &    \ecbox{7}{\Aone} \ar[rd] \ar[d]    &    \ecbox{7}{\Bl_2\Ptwo} \ar[d] \ar[rd]   &       &       &       &     \\
   8   &       &    \ecbox{8}{\FF_2}    &    \ecbox{8}{\Bl_1\Ptwo} \ar[d]    &    \ecbox{7}{\Pone \times \Pone}    &       &       &     \\
   9  &            &            &          \ecbox{9}{\Ptwo}   &            &    	    &    	    &  
    }
  \end{equation*}
  \caption{Generalised del Pezzo surfaces $\tS$ in increasing degree with
    $\#\{\text{negative curves on $\tS$}\} \le \rk\Pic(\tS) + 1$. The boxed
    ones are exactly the equivariant compactifications of $G_d$ for some
    $d$. The dashed ones are exactly the equivariant compactifications of a
    homogeneous space for $G_d$ for some $d$.  Arrows denote blow-up maps (in
    degree $\ge 4$, only maps used in our proofs are included).}
  \label{fig:blow-ups}
\end{figure} 

We note that for each of the types of degree \emph{at most three} given in
Figure~\ref{fig:blow-ups}, there is a \emph{unique} surface over $K$ with that
type, up to isomorphism. Indeed, for the surfaces of degree three this follows
from the classification given in \cite{MR80f:14021}. This also implies
uniqueness for all surfaces of degree greater than three, except for perhaps
the quartic del Pezzo surface of type $\Athree$ with four lines. However, we
also have uniqueness in this case on noticing that such a surface is obtained
by contracting a unique $(-1)$-curve on a del Pezzo surface of degree three
and type $\Afour$. There is again a unique surface of this type by
\cite{MR80f:14021}. Note that this result does not hold for some of the lower
degree surfaces in Figure~\ref{fig:blow-ups}, for example there are infinitely
many generalised del Pezzo surfaces of degree two and type $\Dsix$ up to
isomorphism (see \cite[Theorem~5.7]{MR1933881}).

Next, it follows from Lemma \ref{lem:blow_down} that we need only consider the
``extremal" surfaces in Figure~\ref{fig:blow-ups}, namely if a surface is an
equivariant compactification of (a homogeneous space for) some $G_d$, then so is
any surface that lies below it in Figure~\ref{fig:blow-ups}.  Conversely, if
a surface is not an equivariant compactification of (a homogeneous space for)
$G_d$, then no surface in Figure~\ref{fig:blow-ups} that lies above it is
either.

We now proceed to classify the generically transitive actions of $G_d$ on some
of the surfaces in Figure~\ref{fig:blow-ups} up to equivalence, for each $d
\in \ZZ$. We briefly outline the method that we will use. Suppose that
$\rho:\tS \to \Ptwo$ is the composition of $r \leq 6$ blow-ups of $\Ptwo$ and
that $\tS$ admits a generically transitive action of $G_d$ for some $d$. Then
by Lemma~\ref{lem:blow_down}, we obtain a generically transitive action of
$G_d$ on $\Ptwo$ in such a way that $\rho$ is $G_d$-equivariant. Also in every case
we will consider, we will be able to choose $\rho$ in such a way that the
line $\{z=0\}$ and the point $(1:0:0)$ are images of negative curves on $\tS$.
As the negative curves on $\tS$ are invariant under the action (see the proof
of Lemma~\ref{lemma:candidates}), the line $\{z=0\}$ and the point $(1:0:0)$
must also be invariant under the action on $\Ptwo$, and hence the action will
have the form given by Lemma~\ref{lem:Ptwo}.

Therefore, we are reduced to the following question: for which of the
actions given in Lemma~\ref{lem:Ptwo} is the map $\rho$ $G_d$-equivariant? This is
equivalent to asking whether the inverse of $\rho$ is an $G_d$-equivariant
birational map $\rho^{-1}:\Ptwo \dashrightarrow \tS$. Also, by
Proposition~\ref{prop:desing} this is again equivalent to asking whether or
not $\pi \circ \rho^{-1}$ is $G_d$-equivariant, where $\pi:\tS \to S$ denotes the
map to the associated singular del Pezzo surface. As $r \leq 6$ however, we
see that $S \subset \PP^{9-r}$ and moreover the map $\pi \circ \rho^{-1}$ is
given by choosing a basis for some linear series $V \subset H^0(\Ptwo,
\mathcal{O}_{\Ptwo}(3))$.  We may therefore appeal to
Lemma~\ref{lemma:equivariant}, and reduce to determining whether or not $V$ is
invariant under the action of $G_d$ on $H^0(\Ptwo, \mathcal{O}_{\Ptwo}(3))$.

We now show this method in action by considering the extremal surfaces
given in Figure~\ref{fig:blow-ups}, beginning with the one such
surface of degree five.

\begin{lemma}\label{lem:quintic}
  The quintic del Pezzo surface of type $\Afour$ admits a unique structure as
  an equivariant compactification of a homogeneous space for $G_1$ (but none
  for $G_0$). It is not an equivariant compactification of $G_d$ for any $d$.
\end{lemma}

\begin{proof}
  The quintic type $\Afour$ is defined by the equations
  \begin{align*}
    &x_2x_4-x_1^2=x_3x_4-x_0x_1=x_0x_2-x_1x_3\\
    ={}&x_1x_2+x_0^2+x_4x_5=x_2^2+x_0x_3+x_1x_5=0.
  \end{align*}
  The associated rational map from $\Ptwo$ is given by
  \begin{equation*}
    (x:y:z)\mapsto (xz^2: yz^2:y^2z:xyz:z^3:-(y^3+x^2z)).
  \end{equation*}
  This is not defined at $(1:0:0)$, and moreover the line $\{z=0\}$ is
  mapped to the singularity $(0:0:0:0:0:1)$. Therefore the associated
  action on $\Ptwo$ must leave these subvarieties invariant, hence is
  of the form given in Lemma~\ref{lem:Ptwo}. For it to be equivariant,
  the associated linear series of cubic forms must be invariant of the
  action of $G_d$, by Lemma~\ref{lemma:equivariant}.  One can check
  that this happens if and only if $2k_1=3k_2$ (this condition comes
  from the term $-(y^3+x^2z)$). So for some $k\neq 0$, we have
  $(k_1,k_2)=(3k,2k)$. If two of $\alpha_1,\alpha_2,\alpha_3$ are
  non-zero, this leads to $d=k=0$, and the action is not generically
  transitive. If only $\alpha_1 \ne 0$, we have $d=k$, and the action
  is equivalent to $\tau_{d,-2d}$. If only $\alpha_2 \ne 0$, we have
  $d=3k$, and the action is equivalent to $\tau_{3k,2k}$. If only
  $\alpha_3 \ne 0$, we have $d=2k$, and the action is equivalent to
  $\tau_{2k,3k}$.  In any case, the stabiliser of a general point has
  order at least two and so the action is not generically free.
\end{proof}

We now consider the extremal surfaces of degree four. 

\begin{lemma}\label{lem:quartic}
   For quartic generalised del Pezzo surfaces we have the following.
  \begin{itemize}
  \item The surface of type $\Athree+2\Aone$ is an equivariant
    compactification of $G_d$ for all $d \in \ZZ$.
  \item The surface of type $\Dfour$ admits a unique structure as an
    equivariant compactification of $G_2$ (but none for other $G_d$
    with $d\geq 0$).
  \item The surface of type $\Afour$ admits a unique structure as an
    equivariant compactification of a homogeneous space for $G_1$ (but
    none for $G_0$).  It is also not an equivariant compactification
    of $G_d$ for any $d$.
  \item The surface of type $\Athree+\Aone$ admits a unique structure
    as an equivariant compactification of $G_1$ (but none for other
    $G_d$ with $d\geq 0$).
  \item The surface of type $\Athree$ (four lines) is not an
    equivariant compactification of a homogeneous space for $G_d$ for
    any $d$.
   \end{itemize}
\end{lemma}

\begin{proof}
  Type $\Athree+2\Aone$: The surface $S$ is defined by
    \begin{equation*}
      x_0x_1-x_2^2=x_0^2-x_3x_4=0.
    \end{equation*}
    Note that this surface is toric. For each $d \in \ZZ$, the action of $G_d$
    is given by the representation
    \begin{equation*}
      (b,t) \mapsto
      \begin{pmatrix}
        1   & 0 & 0  & 0 & 0 \\
        b^2 & t^{2d} & 2t^{d}b & 0 & 0 \\
        b   & 0 & t^d  & 0 & 0 \\
        0   & 0 & 0  & t & 0 \\
        0   & 0 & 0  & 0 & t^{-1} \\	
      \end{pmatrix},
    \end{equation*}
    which is easily checked to be generically free and generically transitive.

    Type $\Dfour$: The surface $S$ can be defined by
    \begin{equation*}
      x_0x_3-x_1x_4=x_0x_1+x_1x_3 + x_2^2=0.
    \end{equation*}
   The associated rational map from $\Ptwo$ is given by
  \begin{equation*}
    (x:y:z)\mapsto (xz^2: z^3:yz^2:-z(xz+y^2):-x(xz+y^2)).
  \end{equation*}
  The associated action on $\Ptwo$ must therefore fix $\{z=0\}$ and $(1:0:0)$,
  hence has the form given by Lemma~\ref{lem:Ptwo}. By considering the term
  $z(xz+y^2)$, we see that we must have $k_1=2k_2$. Also by considering the
  action on the final term, we see that for the linear series to be invariant
  we must have $\alpha_1=\alpha_3=0$ (this is due to the appearance of the
  monomials $y^3$ and $xyz$ if $\alpha_1$ or $\alpha_3$ are non-zero).
  Therefore we also have $k_1=d$. Such an action may occur only when $d$ is
  even, in which case it is equivalent to $\tau_{d,-d/2}$. This is
  faithful if and only if $|d|=2$.  The action when $d=2$ may be given
  explicitly via the representation
   \begin{equation*}
      (b,t) \mapsto
      \begin{pmatrix}
        t^2   &  b   & 0 & 0    & 0 \\
        0     &  1   & 0 & 0    & 0 \\
        0     &  0   & t & 0    & 0 \\
        0     & -b   & 0 & t^2  & 0 \\
        -bt^2 & -b^2 & 0 & bt^2 & t^{4} \\	
      \end{pmatrix}.
   \end{equation*}

    Type $\Afour$: The surface $S$ can be defined by
    \begin{equation*}
      x_0x_1-x_2x_3=x_0x_4+x_1x_2 + x_3^2=0.
    \end{equation*}
    The associated rational map from $\Ptwo$ is given by
   \begin{equation*}
     (x:y:z)\mapsto (z^3:xyz: xz^2:yz^2:-y(x^2+yz)).
   \end{equation*}
   The associated action on $\Ptwo$ must therefore fix $\{z=0\},(1:0:0)$ and
   $(0:1:0)$. This implies in particular that $\alpha_1=0$.  By considering
   the final term, we see that we must have $2k_1=k_2$ and $\alpha_3=0$ (due
   to a term of the form $x^2z$ if $\alpha_3\neq0$). Such an action is
   therefore equivalent to $\tau_{d,2d}$ for some $d$. This is faithful if and
   only if $|d|=1$, in which case the stabiliser of a generic point has order
   $2$. Explicitly the action for $d=1$ is given by
   \begin{equation*}
      (b,t) \mapsto
      \begin{pmatrix}
        1    &  0     & 0 & 0    & 0 \\
        0     &  t^3 & 0 & bt^2    & 0 \\
        b     &  0     & t & 0    & 0 \\
        0     & 0       & 0 & t^2  & 0 \\
        0   & -2bt^3 & 0 & -b^2t^2 & t^{4} \\	
      \end{pmatrix}.
   \end{equation*}
   
   Type $\Athree+\Aone$: Note that we originally considered this
   surface in \cite[Section~5]{MR2753646}.  The equations are given by
    \begin{equation*}
      x_1x_3-x_2^2=x_0x_3+x_2x_4 + x_0^2=0.
    \end{equation*}
    The associated rational map from $\Ptwo$ is given by
   \begin{equation*}
     (x:y:z)\mapsto (xyz:y^3: y^2z:yz^2:-xz(x+z)).
   \end{equation*}
   The action on $\Ptwo$ must fix $\{y=0\},\{z=0\}$ and $(0:0:1)$. Hence we
   must have $\alpha_2=\alpha_3=0$.  The linear series is invariant if and
   only if $k=-d$, in which case this action is equivalent to
   $\tau_{d,d}$. This is faithful if and only if $|d|=1$ and the action in the
   case $d=-1$ is given in \cite[Section~5]{MR2753646}.

   Type $\Athree$: This is given by the equations
  \begin{equation*}
    x_0x_1-x_2^2=(x_0+x_1+x_3)x_3-x_2x_4=0.
  \end{equation*}
  It is described in \cite[Section~6.4]{thesis}.  
  The associated rational map from $\Ptwo$ is given by
  \begin{equation*}
    (x:y:z) \mapsto (z^3:x^2z:xz^2:xyz-z^3:(x+y)(xy-z^2)).
  \end{equation*}
  The action on $\Ptwo$ must therefore fix $(1:0:0)$, $(0:1:0)$, $(1:-1:0)$
  and the lines $\{x=0\}$, $\{z=0\}$. Using Lemma~\ref{lem:Ptwo}, we must have
  $\alpha_1=\alpha_2=0$ and moreover $k_1=k_2=d$, as there are three fixed
  points. Considering the term $(x+y)(xy-z^2)$, we deduce that the linear
  series is invariant only if $d=0$, which does not give a generically
  transitive action.
\end{proof}

Finally, we consider the cubic surfaces. 

\begin{lemma}\label{lem:cubic}
  For cubic generalised del Pezzo surfaces we have the following.
  \begin{itemize}
  \item The surface of type $\Esix$ admits a unique structure as an
    equivariant compactification of a homogeneous space for $G_2$ (but
    none for $G_0$ or $G_1$).
  \item The surface of type $\Afive+\Aone$ admits a unique structure
    as an equivariant compactification of a homogeneous space for
    $G_1$ (but none for $G_0$).
   \end{itemize}
   Moreover given any generically transitive action of $G_d$ on these
   surfaces, any fixed point that lies on a $(-1)$-curve must also lie on a
   $(-2)$-curve.
\end{lemma}

\begin{proof}
    Type $\Esix$: This is defined by $x_3x_0^2-x_0x_2^2+x_1^3=0$. It is the
    closure of the image of $\Ptwo$ under the rational map
    \begin{equation*}
      (x:y:z) \mapsto (z^3: yz^2 : xz^2 : x^2z-y^3),
    \end{equation*}
    with $(1:0:0)$ and $\{z=0\}$ in $\Ptwo$ fixed.  The only questionable part
    of the linear series is $x^2z-y^3$, which maps to an element of the linear
    series under the matrix in Lemma~\ref{lem:Ptwo} if and only if $2k_1=3k_2$
    and $\alpha_1=\alpha_3=0$. So there is an integer $k$ with
    $(k_1,k_2)=(3k,2k)$.  This gives an action that is equivalent to
    $\tau_{2k,3k}$, which is faithful if and only if $|k|=1$.  In this case,
    the stabiliser of a general point has order $3$, so the action is not
    generically free. The induced action on the surface in the case $k=1$ is
    given by
    \begin{equation*}
      (b,t) \mapsto
      \begin{pmatrix}
        1    &  0     & 0 & 0     \\
        0     &  t^2 & 0 & 0     \\
        b     &  0     & t^3 & 0    \\
        b^2     & 2bt^3       & 0 & t^6 \\
      \end{pmatrix}.
   \end{equation*}
    On the only line $\{x_0=x_1=0\}$, the only fixed point is the singularity
    $(0:0:0:1)$.

    Type $\Afive+\Aone$: This surface has equations
    $x_1^3+x_2x_3^2+x_0x_1x_2=0$. The determination of all actions 
    of any $G_d$ on this surface was given in \cite[Lemma~4]{arXiv:1205.0373},
    but we reprove this result for completeness. The associated rational map is
    \begin{equation*}
      (x:y:z) \mapsto (-z^3-x^2y: yz^2: y^2z: xyz).
    \end{equation*}
    An associated action of $G_d$ on $\Ptwo$ must fix the points $(1:0:0),
    (0:1:0)$ and the lines $\{y=0\}, \{z=0\}$. In the form of
    Lemma~\ref{lem:Ptwo}, we must have $\alpha_1=\alpha_3=0$ and $k_1=d$.  The
    associated linear series is invariant if and only if $2k_1=k_2$, in which
    case we obtain and action on $\Ptwo$ that is equivalent to
    $\tau_{d,-2d}$.  This action is faithful if and only if $|d|=1$, in which
    case the stabiliser of a general point has order $2$. The induced action
    on $S$ in the case $d=1$ is given by
    \begin{equation*}
      (b,t) \mapsto
      \begin{pmatrix}
        t^{4}   &  -b^2t^{2}     & 0  & -2bt^3     \\
        0     &  t^2 & 0 & 0     \\
        0     &  0     & 1 & 0    \\
       0     &  bt^2       & 0 & t^3 \\
      \end{pmatrix}.
   \end{equation*}
    On the only lines $\{x_1=x_2=0\}$ and $\{x_1=x_3=0\}$, the fixed points are the
    singularities $(0:0:1:0)$ and $(1:0:0:0)$.
\end{proof}

\begin{proof}[Proof of Theorem~\ref{thm:dp}]
  By Lemma~\ref{lem:quartic}, the quartic generalised del Pezzo
  surfaces of types $\Athree +2\Aone,\Dfour$ and $\Athree +\Aone$ are
  equivariant compactifications of $G_d$ for some $d$. Therefore, all
  surfaces below them in Figure~\ref{fig:blow-ups} also are by
  Lemma~\ref{lem:blow_down}. This is exactly the first collection of
  surfaces given in the statement of Theorem~\ref{thm:dp}.

  Next, by Lemma~\ref{lem:cubic} we see that the cubic surfaces of
  types $\Esix$ and $\Afive +\Aone$ are equivariant compactifications
  of homogeneous spaces for $G_d$ for some $d$.  Again by
  Lemma~\ref{lem:blow_down}, we deduce that all surfaces below them in
  Figure~\ref{fig:blow-ups} are also equivariant compactifications of
  homogeneous spaces for $G_d$ for some $d$. Also, by
  Lemma~\ref{lem:quintic} and Lemma~\ref{lem:quartic}, we know that
  the quintic generalised del Pezzo surface of type $\Afour$ and the
  quartic generalised del Pezzo surface of type $\Afour$ are not
  equivariant compactifications of $G_d$ for any $d$.  In particular,
  this implies the same result for every surface lying above them in
  Figure~\ref{fig:blow-ups} by Lemma~\ref{lem:blow_down}.

  To complete the proof of Theorem~\ref{thm:dp}, it suffices to show
  that the remaining surfaces in Figure~\ref{fig:blow-ups} are not
  equivariant compactifications of homogeneous spaces for $G_d$ for
  any $d$. For the quartic surface of type $\Athree$ with four lines,
  this follows from Lemma~\ref{lem:quartic}. The cubic del Pezzo
  surfaces of types $\Dfive$ and $\Afive$ have one-dimensional
  automorphism groups by \cite[Table~3]{MR2584614}, so they cannot
  have a generically transitive action of any $G_d$.  Surfaces of type
  $\Eseven$ and $\Aseven$ of degree $2$ are blow-ups of the cubic
  surfaces of type $\Esix$ and $\Afive+\Aone$ in a point on one of the
  $(-1)$-curves outside the $(-2)$-curves. However by
  Lemma~\ref{lem:cubic}, there are no generically transitive actions
  fixing such points and hence these surfaces of degree $2$ cannot
  have such an action.  Finally, surfaces of type $\Dsix+\Aone$ and
  $\Dsix$ in degree $2$ and type $\Eeight$ in degree $1$ are blow-ups
  of surfaces that have no generically transitive action of $G_d$, so
  they also cannot have such an action. This completes the proof of
  Theorem~\ref{thm:dp}.
\end{proof}

\bibliographystyle{alpha}

\end{document}